\documentclass[11pt,a4paper]{amsart}
\usepackage{todonotes} 
\usepackage[pagebackref,pdfauthor={Yiannis N. Petridis, Morten S. Risager},
pdftitle={Averages over Heegner points in the hyperbolic circle problem},
            pdfkeywords={Hyperbolic lattice points, Heegner points},
             pdfcreator={Pdflatex}]
{hyperref}
\hypersetup{colorlinks=true}
    
\usepackage{amsfonts,amssymb,amsmath,amsthm}

\usepackage{enumitem}

\newtheorem{theorem}{Theorem}[section]
\newtheorem{lem}[theorem]{Lemma}
\newtheorem{prop}[theorem]{Proposition}

\theoremstyle{definition}

\newtheorem{remark}[theorem]{Remark}
\newtheorem*{thankyou}{Acknowledgements}
\numberwithin{equation}{section}

\def \e {{\varepsilon}}

\def \g {{\gamma}}
\def \G {{\Gamma}}
\def \R {{\mathbb R}}
\def \H {{\mathbb H}}

\def \GmodH {{\Gamma\backslash\H}}

\def \psl  {{\hbox{PSL}_2( {\mathbb R})} }
\def \pslz  {{\hbox{PSL}_2( {\mathbb Z})} }

\newcommand{\norm}[1]{\left\lVert #1 \right\rVert}
\newcommand{\abs}[1]{\left\lvert #1 \right\rvert}
\newcommand{\inprod}[2]{\left \langle #1,#2 \right\rangle}
\newcommand{\vol}[1]{\hbox{vol}( #1 )}

\title[Heegner points, hyperbolic circle problem]{Averaging over Heegner points in the hyperbolic circle problem}    
\author{Yiannis N. Petridis}
\address{Department of Mathematics, University College London, Gower Street, London WC1E 6BT, United Kingdom}
\email{i.petridis@ucl.ac.uk}

\author{Morten S. Risager}
\address{Department of Mathematical
  Sciences, University of Copenhagen, Universitetsparken 5, 2100
  Copenhagen \O, Denmark}
\email{risager@math.ku.dk}

\thanks{The second author was supported by a Sapere Aude grant from
  The Danish Council for Independent Research (Grant-id:0602-02161B)}
\date{\today}

\begin{document}
\begin{abstract}
For $\G=\pslz$ the hyperbolic circle problem aims
to estimate the number of elements
of the  orbit $\G z$ inside the
hyperbolic disc centered at $z$ with
radius  $\cosh^{-1}(X/2)$. 
We show that, by averaging over
Heegner points  $z$ of discriminant $D$, Selberg's error term estimate
can be improved, if $D$ is
large enough. The proof uses bounds on spectral exponential sums, and 
results  towards the sup-norm conjecture of eigenfunctions, and the Lindel\"of conjecture for twists of the $L$-functions attached to Maa{\ss} cusp forms. 
\end{abstract}
\maketitle

\section{Introduction}
For $\G=\pslz$ we consider the standard point-pair invariant on the upper half-plane $\H$ given
by \begin{equation*}u(z,w)=\frac{\abs{z-w}^2}{4\Im(z)\Im(w)}.\end{equation*}   The hyperbolic circle
problem aims to find good estimates, as $X\to\infty$,  on 
\begin{equation}\label{error}N(z,w,X)-\frac{\pi  X}{\vol{\GmodH}},\end{equation}
where \begin{equation*}N(z,w,X)=\#\{\gamma\in \Gamma; 4u(\g z,w)+2\leq X\}.\end{equation*} 
The best known estimate in this direction is due to Selberg who proved 
that 
\begin{equation}
  \label{Selberg-bound}
 N(z,w,X)-\frac{\pi
    X}{\vol{\GmodH}}=O(X^{2/3}).
\end{equation}
 The corresponding  estimate for the error term of the hyperbolic
 circle problem  modified to exclude small eigenvalues is also known
 for any cofinite group  and has not been improved for any group or any $z$, $w$.  The conjectural bound for the error term $O(X^{1/2+\e})$ would be
optimal except for $X^\e$ possibly being replaced by powers of
$\log X$, see \cite{PhillipsRudnick:1994a}.

We investigate averages of \eqref{error} with $z=w$ when  $z$ averages over the set of
Heegner points:
Consider the $\G$-orbits of binary quadratic forms \begin{equation*}ax^2+bxy+cy^2\end{equation*} of
negative fundamental discriminant $b^2-4ac=D<0$, and with $a>0$. Given such an orbit one associates the
corresponding Heegner
point in $\H$ given by
\begin{equation*}z=\frac{-b+i\sqrt{\abs{D}}}{2a}.\end{equation*} 
Here
the point of the orbit is chosen such that $z$ lies in the usual
fundamental domain of $\G$. Denote  the set of Heegner
points of  discriminant $D$ by $\Lambda_D$ and the class number by
$h(D)=\#\Lambda_D$. The order of growth of $h(D)$ is controlled by the estimates
\begin{equation}\label{siegel}
\abs{D}^{1/2-\e}\ll_\e h(D)\ll\abs{D}^{1/2}\log |D|,
\end{equation}
where the lower bound is a strong but ineffective result of Siegel,
see e.g. \cite[Ch. 21]{Davenport:1980}. 

Duke \cite{Duke:1988} showed  that Heegner points get equidistributed
on $\GmodH$, i.e. for $f$ a smooth compactly supported function
on $\GmodH$ we have, as $D\to\infty$,
\begin{equation}\label{equidistribution}
  \frac{1}{h(D)}\sum_{z\in
    \Lambda_D}f(z)\to\frac{1}{\vol{\GmodH}}fd\mu(z).\end{equation}  
In \cite[Thm 1.1
]{PetridisRisager:2016a} we have improved on Selberg's bound \eqref{Selberg-bound} when we
average the center locally in $\GmodH$. Therefore, we may also suspect an
improvement in \eqref{Selberg-bound} if we make a \emph{discrete} average over
Heegner points. Our main theorem confirms that this is indeed the case.
\begin{theorem}\label{maintheorem} Let $f$ be a smooth compactly
  supported non-negative function on $\GmodH$. Then
  \begin{equation*}
  \frac{1}{h(D)}\sum_{z\in \Lambda_D}f(z)\left(N(z,z,X)-\frac{\pi X}{\vol{\GmodH}}\right)=O(X^{7/12+\varepsilon}+X^{4/5+\varepsilon}D^{-4/165+\varepsilon}).
\end{equation*}
\end{theorem}
\begin{remark}
We notice that for $D\geq X^{11/2+\varepsilon}$ this is better than
what we would get using Selberg's bound \eqref{Selberg-bound}.
\end{remark}

Let $\lambda_j=1/4+t_j^2$, $t_j\geq 0$, $j\ge 1$ be
the cuspidal eigenvalues of the automorphic Laplacian on $L^{2}(\GmodH)$ listed according to
  multiplicity, and $u_j$ the corresponding Hecke-Maa{\ss} cusp forms
  normalized to have $L^2$-norm equal to one.
We remark that the cuspidal eigenvalues satisfy   Weyl's law
\begin{equation}\label{weyllaw}
  \sum_{\abs{t_j}\leq T}1=\frac{\vol{\GmodH}}{4\pi}T^2+O(T\log T).
\end{equation}
We also set $\lambda_0=0$, $t_0=i/2$ to correspond to the constant
eigenfunction $(\vol{\GmodH})^{-1/2}.$ 

 Let $L(u_j\times \chi_D,1/2)$ be
  the central value of the $L$-function of $u_j$ twisted by the
  odd primitive quadratic character $\chi_D$ with conductor $\abs{D}$. 

Let $u$ be a Maa{\ss} cusp form with eigenvalue $1/4+t^2$ or an Eisenstein series $E(z, 1/2+it)$. In our proof of Theorem \ref{maintheorem} we are using
approximations to the following three conjectures:
\begin{enumerate}[label=(\text{C\arabic*})]
\item \label{Lindelof-conjecture} The Lindel\"of conjecture for $L(u\times \chi_D,1/2)$:
  \begin{equation*}
    L(u\times \chi_D,1/2)=O(((1+\abs{t})D)^\varepsilon).
  \end{equation*}
\item \label{sup-norm-conjecture} The sup-norm conjecture, i.e.
  \begin{equation*}
   \sup_{z\in K}|u(z)|=O_K((1+\abs{t})^\varepsilon)
  \end{equation*}
for any compact set $K$.
\item \label{spectral-average-conjecture} Bounds on spectral
  exponential sums
  \begin{equation*}
 \sum_{\abs{t_j}\leq
   T}X^{it_j}=O(X^\varepsilon(1+T)^{1+\varepsilon})\label{spectral-exp-sum-conjecture},
\end{equation*}
see \cite{PetridisRisager:2016a}.
\end{enumerate}
 
  Assuming these three conjectures we may improve Theorem \ref{maintheorem}
and prove the following conditional result: 
\begin{theorem}\label{maintheorem-conditional} Let $f$ be a smooth compactly
  supported non-negative function on $\GmodH$. Then assuming
  \ref{Lindelof-conjecture}, \ref{sup-norm-conjecture}, and
  \ref{spectral-average-conjecture} we have
  \begin{equation*}
  \frac{1}{h(D)}\sum_{z\in \Lambda_D}f(z)\left(N(z,z,X)-\frac{\pi X}{\vol{\GmodH}}\right)=O(X^{1/2+\varepsilon}+X^{4/5+\varepsilon}D^{-1/10+\varepsilon}).
\end{equation*}
\end{theorem}
\begin{remark}
  Theorem \ref{maintheorem-conditional} implies that if $D\geq
  X^{4/3+\varepsilon}$ then this is better than
what we would get using Selberg's bound
   \eqref{Selberg-bound}, and if $D\geq
  X^{3}$ we have the bound $O(X^{1/2+\varepsilon})$.
\end{remark}

In order to prove the unconditional bound in Theorem \ref{maintheorem}
we use approximations to  \ref{Lindelof-conjecture}, \ref{sup-norm-conjecture}, and
  \ref{spectral-average-conjecture}. For an approximation to
  \ref{Lindelof-conjecture} we use a recent theorem by Young \cite{Young:2014b} on the cubic moment of the $L$-function over short sums, see Theorem \ref{young} below. To adress \ref{spectral-average-conjecture} we
  use the following estimate due to Sarnak and Luo \cite[Eq. (58)]{LuoSarnak:1995}: 
  \begin{equation}\label{oscillation}
    \sum_{\abs{t_j}\leq T}X^{it_j}=O(X^{1/8}T^{5/4+\varepsilon}).
  \end{equation}
As an approximation to \ref{sup-norm-conjecture} we prove the
following average bound, which may be of independent interest:
\begin{theorem}\label{sup-norm-average-intro} We have 
  \begin{equation*}
 T^{-2}\sum_{1\leq t_j\leq T}\norm{u_j}_{\infty}^2=O(T^{2(1/2-1/8)+\varepsilon}).   
  \end{equation*}
 \end{theorem}
\begin{remark}
This shows that on average
$\norm{u_j}_{\infty}$ is of size $O(\abs{t}^{3/8+\varepsilon})$.  This is an improvement on average of  the  individual bound $\norm{u_j}_\infty=O(\abs{t}^{5/12+\varepsilon})$ due to Iwaniec and Sarnak  \cite[Eq. (A.15)]{iwaniecsarnak:1995a}. 
The individual bound $\norm{u_j}_{\infty}=O(\abs{t}^{3/8+\varepsilon})$
 would follow if we could prove expected lower bounds on the mean square for the Hecke eigenvalues, see \cite[Remark 1.6]{iwaniecsarnak:1995a}.
  Moreover, 
  Young \cite{Young:2015a} proved the following bound   for the sup-norm of Eisenstein series on a compact set $K$:  \begin{equation} \label{Youngs-bound}
    \norm{E(z,1/2+it)}_{\infty,K}=O_{K,\varepsilon}((1+\abs{t})^{1/2-1/8+\varepsilon}).
  \end{equation}

\end{remark} 

\begin{remark}
The situation is much easier when we fix $w$ and average over $z\in \Lambda_D$.  We get the following estimate:
\begin{equation}\label{optimalerror}\frac{1}{h(D)}\sum_{z\in \Lambda_D}\left(N(z,w,X)-\frac{\pi
    X}{\vol{\GmodH}}\right)= O(X^{1/2+\e})\end{equation} for
$D\geq X^{3}$. On GRH we have \eqref{optimalerror} for
$D\geq X^1.$
We omit the proof. It is simpler than the proof of Theorem \ref{maintheorem}, because we can separate the average over $z\in\Lambda_D$ for $u_j(z)$ and $E(z, 1/2+it)$ from $u_j(w)$ and $E(w, 1/2+it)$ in the pre-trace formula.
\end{remark}

\begin{thankyou} We thank Gergely Harcos and P\'eter Maga for useful comments on
  sup-norm bounds.  
\end{thankyou}

\section{Equidistribution of Heegner points}\label{Heegner}

Duke's proof of the equidistribution of Heegner points \eqref{equidistribution} involves non-trivial bounds on Fourier coefficients of half-integral
weight Maa{\ss} forms and uses the Kuznetsov formula. The technique is due to Iwaniec
\cite{Iwaniec:1987a}, who proved non-trivial bounds for the Fourier coefficients of holomorphic forms of half-integral weight by using the Petersson formula to relate them to Kloosterman sums $K(m, n, c)$ and exhibited  cancellations for sums of these as  $c$ varies.  Duke  proved the following bounds
on the average of eigenfunctions, the so-called \lq Weyl sums\rq, see \cite[p. 89]{Duke:1988}:
\begin{align*}
 \frac{1}{h(D)}W_c(D,t_j)&=\frac{1}{h(D)}\sum_{z\in\Lambda_D}u_j(z)\ll_{\varepsilon}\abs{t_j}^AD^{-1/28+\varepsilon},\\
\frac{1}{h(D)}W_E(D,t)&=\frac{1}{h(D)}\sum_{z\in\Lambda_D}E(z,1/2+it)\ll_{\varepsilon}\abs{t}^AD^{-1/28+\varepsilon}.
\end{align*}
One can then use standard approximation techniques to prove
\eqref{equidistribution}.

To improve on these bounds we use  that Weyl sums are connected  to $L$-functions. We assume that the Maa{\ss} cusp forms $u_j$ are also
eigenfunctions of all Hecke operators. 
 We quote
\cite[Eq. (2.2)]{Young:2014b} for the following Waldspurger--Zhang type formula: 
\begin{equation}\label{Waldspurger-Zhang}
\lvert W_c(D,t_j)\rvert^2=\frac{\sqrt{\abs{D}}L(u_j\times\chi_D,1/2)L(u_j,1/2)}{2L(\hbox{sym}^2u_j,1)}.
\end{equation} 
Similarly we have\begin{equation} \label{average-Eisenstein}
  W_E(D,t)=\left(\frac{{\sqrt{\abs{D}}}}{2}\right)^{1/2+it}\frac{L(1/2+it,
    \chi_{D})\zeta(1/2+it)}{\zeta(1+2it)},
\end{equation}
  see \cite[Eq. 22.45]{IwaniecKowalski:2004a}.

Young recently proved the following bound:

\begin{theorem} {\cite[Thm. 1.1]{Young:2014b}}\label{young}  For the third moment of twists of  the $L$-functions of  Maa{\ss} forms we have
  \begin{equation*}
    \sum_{T\leq t_j\leq
      T+1}L(u_j\times\chi_D,1/2)^3+\int_{T}^{T+1}\abs{L(1/2+it,\chi_D)}^6\ll (\abs{D}(1+T))^{1+\e}.
  \end{equation*}
\end{theorem}
In our formulation we have used
the positivity of central values, see \cite[Thm 1]{JacquetChen:2001a} to restrict to  Maa\ss{} forms for
the full modular group.

\begin{lem}\label{short-average-bound-on-weyl}The following estimates on short averages of the Weyl sums hold:
\begin{equation*}
  \frac{1}{h(D)^2}\sum_{T\leq t_j\leq
    T+1}\lvert W_c(D,t_j)\rvert^2\ll D^{-\frac{1}{6}+\e}(1+T)^{1+\e},
\end{equation*}
\begin{equation*}
\frac{1}{h(D)^2}\int_T^{T+1}\lvert W_E(D,t)\rvert^2dt\ll D^{-\frac 1 6 +\e}(1+T)^{1+\e}.
\end{equation*}

\end{lem}
\begin{proof}
Ivi\'c  proved \cite{Ivic:2001} that
\begin{equation}
  \label{ivic}
  \sum_{T\leq t_j\leq
      T+1}L(u_j,1/2)^3=O((1+T)^{1+\e})
\end{equation}
and the bound  
\begin{equation}
  \label{weyl}
  \int_{T}^{T+1}\abs{\zeta(1/2+it)}^6dt=O((1+T)^{1+\e})
\end{equation}
follows from the classical Weyl estimate on the Riemann zeta function on the critical line, see e.g. \cite[Theorem 5.5]{Titchmarsh:1986a}.

We will also need lower bounds
\begin{equation}\label{lower-symmetricsquare}
  L(\hbox{sym}^2u_j,1)\gg_\e \abs{t_j}^{-\epsilon},\textrm{ and } \zeta(1+it)\gg \log\log\abs{t}/\log\abs{t}.
\end{equation}
The first bound is due to Lockhart and Hoffstein \cite[Thm
0.2]{HoffsteinLockhart:1994}, and the second is classical
\cite[Thm. 5.17]{Titchmarsh:1986a}.

We use first \eqref{Waldspurger-Zhang} and
\eqref{average-Eisenstein} and the H\"older inequality with exponents $(1/3,1/3,1/3)$ so that we can apply Theorem \ref{young}.  With the help of   \eqref{ivic}, \eqref{weyl},
\eqref{lower-symmetricsquare}, and  \eqref{weyllaw} we get
\begin{equation*}
  \sum_{T\leq t_j\leq
    T+1}\lvert\sum_{z\in\Lambda_D}u_j(z)\rvert^2\ll (D^{\frac 5 6 }(1+T))^{1+\e},
\end{equation*}
\begin{equation*}
\int_T^{T+1}\lvert\sum_{z\in\Lambda_D}E(z,1/2+it)\rvert^2dt\ll (D^{\frac 5 6 }(1+T))^{1+\e}.
\end{equation*}
Using the lower bound in \eqref{siegel} for $h(D)$, we establish  the claim.
\end{proof}

\begin{remark}\label{bound-on-GRH}
We notice that on GRH, or more precisely  \ref{Lindelof-conjecture}, we may replace $D^{-1/6+\varepsilon}$ by $D^{-1/2+\varepsilon}$ in both bounds.

\end{remark}

\section{Sup-norm estimates}
The general 
bound for the sup-norm is  
\begin{equation}\label{convexity-bound}
  \norm{u_j}_\infty=O((1+\abs{t_j})^{1/2}),
\end{equation}
see e.g. \cite{SeegerSogge:1989a} or Theorem \ref{weight-k-convexity} below.
This is sometimes called the convexity bound for sup-norms. 
In a ground-breaking work  Iwaniec and Sarnak \cite{iwaniecsarnak:1995a} showed that this bound may be improved 
for $\pslz$ to
\begin{equation}\label{iwaniec-sarnak}
  \norm{u_j}_\infty=O((1+\abs{t_j})^{1/2-1/12+\varepsilon}).
\end{equation}
Denote by $\norm{u}_{\infty,K}$  the supremum of a function $u$
  restricted to the set $K$. They also conjectured that if we restrict  to a compact set $K$ the sup-norm
is essentially bounded
  \begin{equation*}\norm{u_j}_{\infty,K}=O((1+\abs{t_j})^{\varepsilon}).\end{equation*} 
In \cite{iwaniecsarnak:1995a} the restriction  to compact sets is not
explicit but is now known to be needed, see \cite{Sarnak:2004c}.  
  More precisely we have 
\begin{equation*}\norm{u_j}_\infty\geq C(1+\abs{t_j})^{1/6-\varepsilon}.\end{equation*} This is a
purely analytic fact that follows from the properties of $K$-Bessel functions.

After the work of Iwaniec and Sarnak there has been a lot of results about  subconvexity of the  sup-norm in the level aspect \cite{BlomerHarcosMichel:2007}, weight
aspect \cite{DasSengupta:2013}, for holomorphic forms, hybrid bounds \cite{Templier:2015, Saha:2015}, as well as
other groups \cite{BlomerMaga:2016a, Marshall:2014}. The bound \eqref{iwaniec-sarnak} has not been improved.

We now discuss average bounds on sup-norms. In particular we prove Theorem \ref{sup-norm-average-intro}, which states that on average in a window of size $T$ we can improve
\eqref{iwaniec-sarnak} to
$\norm{u_j}_\infty=O((1+\abs{t_j})^{1/2-1/8+\varepsilon})$. 

The Maa{\ss} cusp forms have Fourier expansions 
\begin{equation*}
  u_j(z)=\sum_{n \neq 0}\rho_j(n)\sqrt{y}K_{it_j}(2\pi
  \abs{n}y)e^{2\pi in x}.
\end{equation*}
If $u_j$ is also a Hecke eigenform with Hecke
eigenvalues $\lambda_j(n)$,   then
$\rho_j(n)=\rho_j(\hbox{sign}(n))\lambda_j(|n|)$. 
It is known \cite[Prop. 19.6]{DukeFriedlanderIwaniec:2002a} 
that
\begin{equation}\label{strong-bound}
  \sum_{n\leq x}\abs{\lambda_j(n)}^2=O(x^{1+\varepsilon}\abs{t_j}^\varepsilon).
\end{equation}
Moreover, we set $\rho_j(n)=\cosh (\pi t_j/2)v_j(n)$, so that
$v_j(n)=v_j(1)\lambda _j (n)$. It is known \cite{Iwaniec:1990a, HoffsteinLockhart:1994}  that 
\begin{equation}
  \label{HL}
\abs{t_j}^{-\varepsilon}\ll_\varepsilon\abs{v_j(1)}\ll_\varepsilon \abs{t_j}^{\varepsilon}.
\end{equation}
\begin{proof}[Proof of Theorem \ref{sup-norm-average-intro}]
This is an adaptation of the proof for individual bounds in \cite{iwaniecsarnak:1995a}.
  We quote \cite[p.~678]{BlomerHolowinsky:2010} for the following
  crucial inequality: Assume that $T\leq t_j\leq T+1$. Then
\begin{align}
    \label{crucial-inequality}
    \nonumber\abs{u_j(z)}^2&\abs{\sum_{l\leq
        L}\alpha_l\lambda_j(l)}^2\ll_\varepsilon
    (LT)^\varepsilon \left(T\sum_{l\leq
      L}\abs{\alpha_l}^2+(L+y)T^{1/2}\big(\sum_{l\leq L}\abs{\alpha_l}\big)^2\right)
  \end{align}
for every sequence $\alpha_n$. This improves on
\cite[Eq. (A.12)]{iwaniecsarnak:1995a}  by replacing $yL^{1/2}$
by $y$.

  Now we choose a smooth  non-negative function $h$ supported in $[1,2]$ with
integral $\int_{\mathbb R} h(t)dt=1$, and consider $h_N(t)=h(t/N)$. 
By choosing $\alpha_n=h_N(n)\lambda_j(n)\abs{v_j(1)}^2$ and using Cauchy--Schwarz on 
the last sum,  we arrive at
\begin{align}
 \nonumber   \abs{u_j(z)}^2&\abs{\sum_{n\leq
       2N}h_N(n)\abs{v_j(n)}^2}^2\\&\ll_\epsilon
    (TN)^\varepsilon\left(T+(N+y)T^{1/2}N\right)\sum_{n\leq
      2N}h_N(n)^2\abs{v_j(n)}^2\\
\nonumber&\ll_\epsilon (NT)^\epsilon\left(T+(N+y)T^{1/2}N\right)N,
\end{align}
where we have used \eqref{strong-bound} and \eqref{HL}. Luo and Sarnak
\cite[p.~233]{LuoSarnak:1995} proved
Iwaniec' mean Lindel\"of conjecture for the Rankin--Selberg convolution in the spectral aspect,  which allowed them to prove  that 
\begin{equation}\label{sum-of-squares}
  \sum_{n=1}^{2N}
  h_N(n)\abs{v_j(n)}^2=\frac{12}{\pi^2}N+r(t_j,N),
\end{equation}
where the reminder is of size $N^{1/2}$ on average:
\begin{equation}\label{this-is-the-beef}
  \sum_{t_j\leq T}\abs{r(t_j,N)}=O(T^{2+\varepsilon}N^{1/2}).
\end{equation}
When we square the right-hand side of \eqref{sum-of-squares}, the term
$(r(t_j,N))^2$ may be dropped by positivity and we find
\begin{equation}\label{ontheway}
\abs{u_j(z)}^2\ll_\varepsilon N^{-2}\left( (NT)^\epsilon\left(T+(N+y)T^{1/2}N\right)N+\abs{u_j(z)}^2N\abs{r(t_j,N)}\right).
\end{equation}
 We first consider the set $A_N=\{z\in \GmodH; y\leq N\}$. We use the subconvexity
 bound \eqref{iwaniec-sarnak} on $u_j$ on the right-hand side of 
 \eqref{ontheway} to see that
 \begin{equation}
\norm{u_j}_{\infty,A_N}^2\ll_\varepsilon N^{-2}\left( (NT)^\epsilon\left(T+NT^{1/2}N\right)N+T^{2(1/2-1/12)}N\abs{r(t_j,N)}\right).
\end{equation}
 Averaging over $t_j$ we find by \eqref{this-is-the-beef} and Weyl's
 law \eqref{weyllaw} that
 \begin{equation*}
   \sum_{1\leq t_j \leq
     T}\norm{u_j}_{\infty, A_N}^2\ll_{\varepsilon}(NT)^\varepsilon N^{-2}\left(T^3N+N^3T^{5/2}+T^{2+2(1/2-1/12)}N^{3/2}\right).
 \end{equation*}
We choose $N=T^{1/4}$ so that the right-hand side is  $O(T^{3-1/4+\varepsilon})$.

For the complement of the set $A_N$, i.e. for $B_N=\{z\in \GmodH; y> N\}$ we argue as follows:
We have set $N=T^{1/4}$. We bound  
$\norm{u_j}_{\infty, B_N}^2$
  individually as $O(T^{2(1/2-1/8)+\e})$
using the  following   simple upper bound  
\begin{equation*}
\abs{u_j(z)}\ll t_j^\varepsilon\left((t_j/y)^{1/2}+t_j^{1/6}\right),
\end{equation*}see \cite[Lemma A.1$'$]{Sarnak:2004c}. We finish the proof by noticing that
\begin{equation*} \norm{u_j}^2_\infty\leq \max(\norm{u_j}_{\infty, A_N}^2, \norm{u_j}_{\infty, B_N}^2).\end{equation*}
\end{proof}

If we consider sup-norms of averages instead of averages of sup-norms
we have much better bounds. 
\begin{prop}\label{sup-of-averages}
Let $\{u_j\}$ be the $L^2$-normalized Maa{\ss} cusp forms for $\pslz$. Then
  \begin{equation*}
\Big\lVert \sum_{T\leq t_j\leq 2T}\abs{u_j(z)}^2\Big\rVert_\infty=O(T^2).
  \end{equation*}
\end{prop}
\begin{proof} Notice first that, if we restrict $z$ to a compact set $K$ the inequality  in  \cite[Prop. 7.2]{Iwaniec:2002a} gives 
  \begin{equation*}
 \sum_{T\leq t_j\leq 2T}\abs{u_j(z)}^2=O(T^2) ,  
  \end{equation*} where the implied constant depends on $K$ and the
  group, but \emph{not} on $z$.
Even if we do not restrict to a compact set the same bound  holds for
$y\leq 2T$, say, by the same inequality. To bound the average for $y\geq 2T$ we use the decay
properties of the $K$-Bessel function.  We use the integral representation
\begin{equation*}
  K_{it}(y)=\int_0^\infty e^{-y\cosh(v)}\cosh(itv)dv,
\end{equation*} 
see \cite[p.~205]{Iwaniec:2002a}.
Fix $y\ge 0$. For $t$ real we have $|\cosh(itv)|\leq 1$. We use $\cosh(v)\geq 1+v^2/2$   to get 
\begin{equation*}
  \abs{K_{it}(y)}\leq\int_0^\infty e^{-y(1+v^2/2)}dv\leq \frac{\sqrt{\pi}e^{-y}}{\sqrt{2}\sqrt{y}}.
\end{equation*}
 The  bound  
 \begin{equation*}\rho_j(n)=O(e^{\pi  t_j/2}\sqrt{\abs{n}}t_j^{\varepsilon})\end{equation*}
 follows trivially from \eqref{strong-bound} and \eqref{HL}. 
We can now prove  good decay properties for $\abs{u_j}$ when $y$ is large
compared to $t_j$, e.g.  when $y>2T$ and $T\leq t_j\leq 2T$ we have
\begin{equation*}\abs{u_j(z)}=O\left(\sum_{n=1}^\infty e^{\pi
  t_j/2}\sqrt{n}t_j^{\varepsilon}\sqrt{y}\frac{e^{-2\pi n
    y}}{\sqrt{ny}}\right)=O( e^{-T}).\end{equation*} The claim follows using
Weyl's law (\ref{weyllaw}).
\end{proof}
\begin{remark}
  We note that the proof of Proposition \ref{sup-of-averages} is
  much simpler than that of Theorem \ref{sup-norm-average-intro}. The
  only input  is
  the use of the local Weyl law \cite[Prop. 7.2]{Iwaniec:2002a} and bounds
  on the Fourier coefficients that are uniform in $t_j$ and $n$.
\end{remark}

We also need a similar result for the Eisenstein series.
\begin{prop}\label{sup-of-averages-Eisenstein}
Let $K$ be a compact set on $\GmodH$. Then 

    \begin{equation*}
    \norm{\int_{T}^{2T}\abs{E(z,1/2+it)}^2dt}_{\infty,K}=O(T^2).
  \end{equation*}
\end{prop}
\begin{proof}
  This follows directly from \cite[Prop. 7.2]{Iwaniec:2002a}.
\end{proof}

\section{Bounds for weight $k$ eigenfunctions.}
In this section we discuss bounds on weight $k$
eigenfunctions. We need this later when we estimate certain
inner products of eigenfunctions. This is important in order to bound derivatives of
cusp forms and Eisenstein series in the proof of Theorem
\ref{coefficient-bounds} below. 

 For the following discussion we adopt the notation and terminology
 from Fay \cite{Fay:1977a}. In this section we allow $\Gamma$ to be
 any discrete subgroup of $\psl$ unless explicitly stated otherwise.

 Let $k$ be an integer and let ${\mathfrak F}_k$ be the space of functions  $f: \mathbb H\to \mathbb C$ satisfying 
\begin{equation*}f(\gamma z)=\left(\frac{cz+d}{c\bar z+d}\right)^kf(z), \quad \gamma\in \Gamma .\end{equation*}
The $k$-Laplacian is defined by 
\begin{equation*}\Delta_k=y^2\left(\frac{\partial^2}{\partial x^2}+\frac{\partial ^2}{\partial y^2}\right)-2i k y \frac{\partial }{\partial x}.\end{equation*}
We write the eigenvalue equation as $\Delta_k u+s(1-s)u=0$ with $s=1/2+it$.
The raising and lowering operators are defined on ${\mathfrak F}_k$ by
\begin{equation*}K_k=(z-\bar z)\frac{\partial }{\partial z}+k, \quad L_k=(\bar z-z)\frac{\partial}{\partial \bar z}-k.\end{equation*}
It is well-known that if $f\in {\mathfrak F}_k$ is differentiable, then
$K_k f\in {\mathfrak F}_{k+1}$, while  $L_k f\in {\mathfrak F}_{k-1}$. Moreover, $\Delta_{k+1}K_k=K_{k}\Delta_k$, $\Delta_kL_{k+1}= L_{k+1} \Delta_{k+1}$.
It is clear that if $f$ is an eigenfunction of $\Delta_k$ with eigenvalue $\lambda$, i.e. $(\Delta_k+\lambda)f=0$, 
then $K_kf$ (resp. $L_{k}f$) is an eigenfunction of $\Delta_{k+1}$
(resp. $\Delta_{k-1}$) with the same eigenvalue. For $f\in {\mathfrak
  F}_k$ and $g\in {\mathfrak F}_l$ we have the product rules: If $k,l$
are integers, then 
\begin{equation}\label{leibniz}K_{k+l}(fg)=(K_kf)g+f(K_l g), \quad L_{k+l}(fg)=(L_kf)g+f(L_l g).\end{equation}
We use polar
coordinates centered at a point $z_0$, defined through \begin{equation*}\frac{z-z_0}{z-\bar z_0}=\tanh (r/2)e^{i\theta},\end{equation*} with $r=r(z, z_0)$ the hyperbolic distance and $\theta=\theta(z, z_0)\in [0, 2\pi]$. Set $v=\cosh r$. We remark that in polar coordinates $d\mu (z)=\sinh r drd\theta.$ We need the radial expansion of $f$, given in Theorem 1.2 in \cite{Fay:1977a}.
For $f$ an eigenfunction of $\Delta_k$ with eigenvalue $s(1-s)$ on  a disk $r(z, z_0)<R$ we have
\begin{equation}\label{spherical-expansion}f(z)\left( \frac{z-\bar{z_0}}{ z_0-\bar z}\right)^k=\sum_{n\in \mathbb Z}f_n(z_0)P_{s, k}^n(z, z_0)e^{in\theta}\end{equation}
with $P_{s, k}^n(z, z_0)=P_{s, k}^n(r)$ given by \begin{equation*}P_{s, k}^n(r)=\left(\frac{v-1}{v+1}\right)^{|n|/2}\left(\frac{2}{1+v}\right)^s F(s-k_n, s+k_n+|n|, 1+|n|;(v-1)/(v+1)).\end{equation*} Here $F(a, b, c; z)$ is the Gauss hypergeometric function and $k_n=kn/|n|$ for $n\ne 0$ and $k_0=k$. 
We can recover the coefficients $f_n(z_0)$ by the formula
\begin{equation}\label{taylor}
n!f_n(z_0)=n!\overline{\bar f_{-n}(z_0)}=K_{k+n-1}K_{k+n-2}\cdots K_{k+1}K_k f(z_0), \quad n> 0,
\end{equation}
see \cite[Eq. 23]{Fay:1977a}.

In this section we will prove the following lemma, which is crucial in the proof of Theorem  \ref{coefficient-bounds}.
\begin{lem}\label{sup-sup-bounds}
Let $f\in{\mathfrak F}_k$ be an eigenfunction of $\Delta_k$ with
eigenvalue $s(1-s)$ and let $C$ be a compact subset of $\GmodH$.
Then there exists a compact subset $C'$ of $\GmodH$ containing $C$
such that 
\begin{equation*}
\norm{K_kf}_{\infty, C}\ll \abs{s}\norm{f}_{\infty, C'}, \quad \norm{L_kf}_{\infty, C}\ll \abs{s}\norm{f}_{\infty, C'} .
\end{equation*}
\end{lem}

\begin{proof}
We set $h(z)=f(z) (z-\bar{z_0})^k/( z_0-\bar z)^k$.
Since $P^n_{s, k}(z_0, z_0)=0$ for $n\ne 0$ and $P^0_{s, k}(z_0, z_0)=1$, we have $f(z_0)=f_0(z_0)$. By (\ref{spherical-expansion}) we get
\begin{equation*}\int_0^{2\pi} h(z)e^{-in\theta}\,d\theta
=2\pi f_n(z_0)P^n_{s, k}(z, z_0).\end{equation*} 
Let $A=A(r_1, r_2)=\{z; r_1\le r(z,
z_0)\le r_2\}$  be a disc or annulus centered at $z_0$. We multiply
with $\overline{P^n_{s, k}(r)}$, and integrate the radial variable to get
\begin{align*}\int_{A}h(z) e^{-in\theta}\overline{P^n_{s, k}(z, z_0)}\, d\mu (z)=&
\int_{A}f_n(z_0) |{P^n_{s, k}(z, z_0)}|^2\, d\mu (z).\end{align*}
This implies the crucial identity
\begin{equation}\label{crucial}f_n(z_0)=\frac{\int_{A}h(z) e^{-in\theta}\overline{P^n_{s, k}(z, z_0)}\, d\mu (z)}{\int_{A}|P^n_{s, k}(z, z_0)|^2\, d\mu (z)}.\end{equation}
We apply \eqref{crucial} for $n=1$ and choose $A$ to be the disc of radius $c|s|^{-1}$, for a sufficiently small constant $c$ to be chosen.  We get
\begin{equation*}|f_1(z_0)|\le \frac{\norm{h}_{\infty, A}\int_A|P^1_{s, k}(r)|d\mu (z)}{\int_A|P^1_{s, k}(r)|^2d\mu (z)}.\end{equation*}
If we can prove that 
\begin{equation}\label{l1-l2-bounds-for-P}\int_A|P^1_{s, k}(r)|d\mu (z)\ll |s|^{-3}, \quad \int_A|P^1_{s, k}(r)|^2d\mu (z)\gg |s|^{-4},\end{equation}
then, by \eqref{taylor}, we have $|K_kf(z_0)|\ll |s| \norm{f}_{\infty, A}.$ By compactness of $C$, we can find a compact set $C'\subset \mathbb H$ containing $C$ such that
\begin{equation*}\norm{K_kf}_{\infty, C}\ll  |s| \norm{f}_{\infty, C'} .\end{equation*}
To prove  \eqref{l1-l2-bounds-for-P}
we need to study the asymptotics of $P_{s, k}^n(r)$ jointly for $r$
small and $t=\Im (s)\to \infty.$ For simplicity let $n$ be nonnegative.
We can approximate the hypergeometric function $F(a, b, c ;z)$ by its Taylor polynomials, see \cite[Eq. (4.13), (4.14)]{Good:1981b}:
\begin{equation*}
F(a, b, c;z)=\sum_{j=0}^{J-1}\frac{(a)_j(b)_j}{(c)_j j!}z^j+O\left(\left|\frac{(a)_J(b)_Jz^J}{(c)_J J! }\right|\right)
\end{equation*}
uniformly in $a, b, c$ as long as
\begin{equation}\label{good-condition}
|z| \max_{j\ge 0}\left|\frac{(a+j)(b+j)}{(c+j)( j+1)}\right|\le \frac{1}{2}.
\end{equation}
It is easily verified that if $|r|<c|s|^{-1}$, then the condition of \eqref{good-condition} is satisfied and we can apply the Taylor series for $F(a, b, c, z)$ with $J=2$ to get
\begin{equation*}\label{P-short-asymptotics}
P^{n}_{s, k}(r)
=\left(\frac{r}{2}\right)^{n}\left(1+ \left(\frac{(s-k)(s+k+n)}{4(1+n)}-\frac{s}{4} -\frac{n}{12}\right)r^2+O(s^4r^4)     \right),\end{equation*}
cf. \cite[Eq. 17]{Fay:1977a}. We specialize to $n=1$ in the simpler form
\begin{equation*}P^1_{s, k}(r)=(r/2)(1+O_{k}(s^2r^2)).\end{equation*} We have 
\begin{equation}\label{vorueber}\int_0^{c/|s|} r^a \sinh r dr\sim c^{a+2} \frac{1}{(a+2)|s|^{a+2}} , \quad |s|\to\infty.\end{equation}
To investigate \eqref{l1-l2-bounds-for-P} we integrate in the $\theta$
variables. Then we apply \eqref{vorueber} for $a=1$ and $a=3$ to  get  the first inequality in \eqref{l1-l2-bounds-for-P}.
For the second inequality we  apply it for $a=2, 4, 6$. The terms with
$a=4,6$ get multiplied by $|s|^2$ and $|s|^4$ respectively. 
All three terms are of the same order of decay, i.e. $|s|^{-4}$. However, when we take into account the constant in the $O_{k}$ and the powers of $c$, we can  choose $c$ sufficiently small to  make the first term  the dominant term.

Finally we prove $\norm{L_k f}_{\infty, C}\ll \abs{s}\norm{f}_{\infty,
  C'}$:  If $f\in {\mathfrak F}_k$, then $\bar f\in {\mathfrak F}_{-k}$. We only need to observe that  
$L_k=\overline{K_{-k}}$, see \cite[Eq. (3)]{Fay:1977a}.
\end{proof}
\begin{remark}\label{tricky1}
Let $f\in{\mathfrak F}_k$. For $ j=1, 2, \ldots , m$
let $A_j$ be either a lowering or a raising operator such that $A_mA_{m-1}\cdots A_1 f$ makes sense. 
Repeated use of Lemma \ref{sup-sup-bounds}  shows that
\begin{equation*}\label{infinity-bound-maass-operators}
\norm{A_mA_{m-1}\cdots A_1 f}_{\infty, C}\ll
\abs{s}^m\norm{f}_{\infty, C'} .\end{equation*}
\end{remark}
We now consider $L^2$-norms of eigenfunctions. We denote by
$\norm{f}_{2,C}$ the $L^2$-norm of $f$ restricted to $C$. Assume now
that  $\G=\pslz$. Moreover, for the corresponding Eisenstein
series $E_k(z, s)$ of weight $k$ we denote $E_k^Y(z, s)$ the function
$E_k(z, s)-\chi_{[Y, \infty)}(y) (y^s+\phi _k(s)y^{1-s}).$ It is well-known that $\phi_0(s)=\xi(2s-1)/\xi (2s)$, where $\xi(s)$ is the completed Riemann zeta function. 
\begin{lem}\label{tricky} 
Let $C$ be a compactly supported
set. Assume that $C\subset \{z\in \H; \Im (z)\le Y\}.$ Then 
\begin{enumerate}[label=(\roman*)]
\item \label{one-one} for $f\in{\mathfrak F}_k$  an $L^2$-eigenfunction of $\Delta_k$
  with eigenvalue $1/4+t^2$ we have
\begin{align*}\norm{K_k f}_2,&\ll \abs{t}\norm{f}_2, \quad \norm{L_k f}_2\ll \abs{t}\norm{f}_2,
\end{align*}
\item \label{two-two}for a weight $k$ Eisenstein series $E_k(z,s)$ we have 
\begin{align*}
\norm{K_kE_k(\cdot, 1/2+it)}_{2,C}&\ll
  \abs{t}\norm{E^Y_k(\cdot, 1/2+it)}_{2},\\
\norm{L_kE_k(\cdot, 1/2+it)}_{2,C}&\ll
  \abs{t}\norm{E^Y_k(\cdot, 1/2+it)}_{2},
\end{align*}
\end{enumerate}
for $\abs{t}$ large.
\end{lem}

\begin{proof}
The claim in \ref{one-one} follows from \cite[Satz 3.1]{Roelcke:1966a}. 

For \ref{two-two} we note that
\begin{equation*}K_kE_k(z, s)=(s+k)E_{k+1}(z, s), \quad L_kE_k(z, s)=(s-k)E_{k-1}(z, s),\end{equation*}
cf. \cite[Eq. (10.8),  (10.9)]{Roelcke:1966a}. It suffices to consider the $L^2$-norms of $E^Y_{k\pm 1}(z, 1/2+it)$. The Maa{\ss}--Selberg relations for weight $k$ Eisenstein series give
\begin{equation*}\norm{E^Y_{k}(\cdot, 1/2+it)}^2=-\frac{\phi'_k(1/2+it)}{\phi_k(1/2+it)}+O_Y(1),\end{equation*}
see \cite[Lemma 11.2, p.~301]{Roelcke:1966a}. Moreover, see \cite[Eq. (10.26)]{Roelcke:1966a}, we have
\begin{equation}\label{phi-recursion-on-k}\phi_{k+1}(s)=\frac{k+1-s}{k+s}\phi_k(s).\end{equation}  This gives 
\begin{equation*}
\norm{E^Y_{k+1}(\cdot, 1/2+it)}^2=-\frac{\phi'_k(1/2+it)}{\phi_k(1/2+it)}+O_Y(1)=\norm{E^Y_k(\cdot, 1/2+it)}^2+O_Y(1).
\end{equation*}
It follows recursively from (\ref{phi-recursion-on-k}) that  $-\frac{\phi'_k(1/2+it)}{\phi_k(1/2+it)}=-\frac{\phi'_0(1/2+it)}{\phi_0(1/2+it)}+o(1)$, as $\abs{t}\to \infty$. Since 
\begin{equation}-\frac{\phi'_0(1/2+it)}{\phi_0(1/2+it)}= 2\Re \frac{\Gamma'}{\Gamma}(1/2+it) +O\left(\frac{\zeta'(1+it)}{\zeta (1+it)} \right)+O(1),
\end{equation}
the result follows using Stirling's formula, which gives $\Gamma'(s)/\Gamma (s)\sim \log s$,  and Weyl's bound $\zeta'(1+it)/\zeta (1+it)\ll \log t/\log\log t.$
\end{proof}
\begin{remark}\label{tricky2}
Let $f$ be an $L^2$-eigenfunction of $\Delta_k$ or a weight $k$ Eisenstein series $E_k(z,s)$. For $ j=1, 2, \ldots , m$
let $A_j$ be either a lowering or a raising operator such that $A_mA_{m-1}\cdots A_1 f$ makes sense. 
Then repeated use of Lemma \ref{tricky}(i) shows that 
\begin{equation*}
\norm{A_mA_{m-1}\cdots A_1 f}_{2}\ll
\abs{s}^m\norm{f}_{2} ,\end{equation*} if $f$ is an $L^2$-eigenfunction.
Moreover, if $f$ is an Eisenstein series, then $A_mA_{m-1}\cdots A_1 f$ is another Eisenstein series of appropriate weight, times a polynomial of degree $m$ is $s$. A similar argument to Lemma \ref{tricky}(ii) gives
\begin{equation*}
\norm{A_mA_{m-1}\cdots A_1 f}_{2,C}\ll
\abs{s}^m\norm{f^Y}_{2} ,\end{equation*}
 Here $C$ is any compactly supported set.
\end{remark}
For completeness we also state and prove the convexity bound for weight $k$ eigenfunctions:
\begin{theorem}\label{weight-k-convexity}
Let $f\in{\mathfrak F}_k$ be an eigenfunction of $\Delta_k$ with eigenvalue $s(1-s)$ and let $C$ be a compact subset of $\GmodH$.
Then there exists a compact subset $C'$ of $\GmodH$ containing $C$ such that
\begin{equation*}\norm{f(z)}_{\infty, C}\ll_{\Gamma, C} |t|^{1/2}\left(\int_{C'} |f(z)|^2d\mu (z)\right)^{1/2}.\end{equation*}
\end{theorem}
 
\begin{proof}
We imitate the argument in \cite{Sarnak:2004c}. We use \eqref{crucial}
for $n=0$ and apply the Cauchy--Schwarz inequality to get:
\begin{align*}|f(z_0)|&=\left|\frac{\int_{A}h(z) \overline{P^0_{s,
                        k}(z, z_0)}\, d\mu (z)}{\int_A|P^0_{s, k}(z,
                        z_0)|^2\, d\mu (z)}\right|\\
&\le\frac{(\int_A |h(z)|^2\, d\mu (z))^{1/2}(\int_A |P^0_{s, k}(z, z_0)|^2\, d\mu (z))^{1/2}}{{\int_A|P^0_{s, k}(z, z_0)|^2\, d\mu (z)}}\\ &=\frac{(\int_A |f(z)|^2\, d\mu (z))^{1/2}}{{(\int_A|P^0_{s, k}(z, z_0)|^2\, d\mu (z))^{1/2}}}\ll |t|^{1/2}\left(\int_A |f(z)|^2\, d\mu (z)\right)^{1/2},
\end{align*}
 if we can show for some annulus $A$ that
\begin{equation}\label{integrated-asymptotics}
\int_A|P^0_{s, k}(z, z_0)|^2\, d\mu (z)\gg |t|^{-1}.
\end{equation} 
To prove this we 
 need the asymptotic behavior of $P_{s, k}^0(r)$ as  $\Im (s)\to \infty$,  $\Re (s)$ fixed and $r>r_0$, see \cite[Eq. 27]{Fay:1977a}:
\begin{equation}\label{spherical-asymptotic}P^0_{1/2+it, k}(r)=\frac{2}{|t|^{1/2}\sqrt{2\pi \sinh r}}\cos (rt-\pi /4)+O(|t|^{-1}).\end{equation}
Since the asymptotics
in (\ref{spherical-asymptotic}) hold for $r$ away from $0$ it is
convenient to work in an annulus $A=A(r_1, r_2)=\{z; r_1<r(z,
z_0)<r_2\}$ centered at $z_0$. 
We have
\begin{align*}\int_{r_1}^{r_2}|P^0_{s, k}(r)|^2\sinh r
  dr&=\int_{r_1}^{r_2} \left(\frac{4}{|t|2\pi \sinh r}\cos^2(rt-\pi
      /4)+O(|t|^{-3/2} )\right)\sinh r dr\\ 
&=\frac{r_2-r_1}{\pi |t|}+O(|t|^{-3/2}).
\end{align*}
We integrate in polar coordinates to get \eqref{integrated-asymptotics}.

We use the same $r_1$ and $r_2$ for all $z_0\in C$ to get a compact set $K'\subset \mathbb H$ such that
\begin{equation*}\norm{f(z)}_{\infty, C}\ll_{ C} |t|^{1/2}\left(\int_{K'} |f(z)|^2d\mu (z)\right)^{1/2}.\end{equation*}
Finally by compactness we can cover $K'$ by a finite set of $\Gamma$-translates of a compact set $C'\subset \GmodH$.
\end{proof}

\section{Squares of eigenfunctions and Heegner points}
Let $f$ be a smooth compactly supported function on $\GmodH$. 
In order to use the results on equidistribution of Heegner points from
section \ref{Heegner} we need bounds on the coefficients in the
spectral expansion of $f\abs{u_j}^2$ and the similar coefficients coming
from
Eisenstein series. 

Very strong bounds (with precise exponential decay in $t_k$) are known on
$\inprod{\abs{u_j}^2}{u_k}$ (see \cite{Sarnak:1994a, Petridis:1995a,
  BernsteinReznikov:2004}) but unfortunately they do not seem uniform
enough for our purposes. In particular for a given $t_j$ the
bounds only hold for $t_k$ large enough (depending on $t_j$). In this
section we obtain much weaker bounds for similar expressions that
hold uniformly in both $t_j$ and $t_k$.   
\begin{theorem}\label{coefficient-bounds} Let $f$ be a smooth
  compactly supported function on $\GmodH$ with
  support $K$.
 For any $b>0$ we have the bound
 \begin{equation*}
    \inprod{f\abs{\phi_1}^2}{\phi_2}\ll_{b,f}
    \left(\frac{1+\abs{t_1}}{1+\abs{t_2}}\right)^b\norm{\phi_1}_{\infty,K}\norm{\phi_1}_2
    \norm{\phi_2}_2 , 
 \end{equation*}
where, for $j=1,2$, the functions  $\phi_j,$  equal $u_j$ with eigenvalue $1/4+t_j^2$ or
 $E^Y(z,1/2+it_j)$, where $Y$ is chosen such that the set $\{z\in\GmodH; y\leq
Y \}$ contains $K$ in its interior.
\end{theorem}
\begin{proof}
By interpolation it suffices to prove the claim for $b=2n$ where  $n$
is a positive integer.
    We have
  \begin{align*}
    (1/4+t_2^2)^n\inprod{f\abs{\phi_1}^2}{\phi_2}&=\inprod{f\abs{\phi_1}^2}{(-\Delta)^n\phi_2}=\inprod{(-\Delta)^nf\abs{\phi_1}^2}{\phi_2}.
  \end{align*}
  In case $\phi_2=E^Y(\cdot, 1/2+it_2)$ we have used the fact  that on the support of $f$ we have $E^Y (\cdot , 1/2+it_2)=E(\cdot , 1/2+it_2)$.
We
see that the statement follows if we can prove that 
\begin{equation*}
  \norm{\Delta^nf\abs{\phi_1}^2}_2\ll_f
  (1/4+t_1^2)^n\norm{\phi_1}_{\infty, K}\norm{\phi_1}_{2}.
\end{equation*}
But since $\Delta^{n}$ consists of compositions of $n$ copies of
$L_1K_0$ this follows from the Leibniz' rule \eqref{leibniz},
remarks \ref{tricky1} and \ref{tricky2}, and the  compactness of the support of $f$.
\end{proof}

The following theorem makes explicit the rate of equidistribution of
Heegner points with the test function $f\abs{\psi}^2$ for $\psi$ an eigenfunction.

\begin{theorem} \label{propo} Let $\G=\pslz$ and let $f$ be a function on
  $\GmodH$ supported on the compact set $K$. Then 
  \begin{align*}
    \frac{1}{h(D)}
\sum_{z\in
  \Lambda_D}f(z)\abs{\psi_t}^2=&\frac{1}{\vol{\GmodH}}\int_{\GmodH}f\abs{\psi_t}^2d
\mu(z)\\
& +
    O_{f,\varepsilon}(\norm{\psi_t}_{\infty,K} D^{-1/12+\e} (1+\abs{t})^{1+\e}) ,
  \end{align*}
  where either $\psi_t=E(z,1/2+it)$ or $\psi_{t_j}=u_j(z)$.
\end{theorem}
Before proving Theorem \ref{propo} we state and prove a bound on
the $L^2$-norm of  $E^Y(\cdot, 1/2+it)$. This is one of the several places in the proof of Theorem
\ref{propo} where  arithmeticity of the
group enters. Here it enters through the growth of the logarithmic derivative of the
scattering determinant.

\begin{lem}\label{eisenstein-bound-again} Let $\G=\pslz$. Then 
  \begin{equation*}
  \norm{E^Y(\cdot,1/2+it)}_{2}\ll_Y \sqrt{\log(2+\abs{t})}.
\end{equation*}
\end{lem}
\begin{proof}
  This follows from the Maa{\ss}--Selberg relation as in the proof of  \cite[Lemma 6.1]{PetridisRisager:2016a}.
\end{proof}

\begin{proof}[Proof of Theorem \ref{propo}]
  We use the spectral expansion of $f\psi_t^2$ to see that
  \begin{align*}
    f\psi_t^2&-\frac{1}{\vol{\GmodH}}\int_{\GmodH}f\psi_t^2d\mu(z)\\ &=\sum_{t_k\neq
    i/2}\inprod{f\psi_t^2}{u_k}u_k+\frac{1}{4\pi}\int_\R\inprod{f\psi_t^2}{E(\cdot, 1/2+ir)}E(\cdot, 1/2+ir)dr.
  \end{align*} After we average over Heegner points, it suffices to 
  show that
  \begin{equation*}
\sum_{t_k\neq
    i/2}\inprod{f\psi_t^2}{u_k}\frac{W_c(D,t_k)}{h(D)}+\frac{1}{4\pi}\int_\R\inprod{f\psi_t^2}{E(\cdot, 1/2+ir)}\frac{W_E(D,r)}{h(D)}dr
\end{equation*}
is $O_{f,\varepsilon}(\norm{\psi_t}_{\infty,K} D^{-1/12+\e}
\abs{t}^{1+\e})$. We bound the discrete contribution, i.e. the sum, and
notice that the continuous contribution can  be bounded in the same
way. By Cauchy--Schwarz we have  
\begin{align*}
  \left(\sum_{L\leq t_k\leq
  2L}\inprod{f\psi_t^2}{u_k}\frac{W_c(D,t_k)}{h(D)}\right)^2&\leq \sum_{L\leq
                                                t_k\leq
                                                2L}{\abs{\inprod{f\psi_t^2}{u_k}}^2}\sum_{L\leq
t_k\leq 2L}\frac{\abs{W_c(D,t_k)}^2}{h(D)^2}\\
&\ll \sum_{L\leq {t_j}\leq 2L}\abs{\inprod{f\psi_t^2}{u_k}}^2 D^{-1/6+\varepsilon}(1+L)^{2+\varepsilon},
\end{align*}
where we have used Lemma
\ref{short-average-bound-on-weyl}.  We bound the remaining sum in two
different ways: either by using 
Theorem \ref{coefficient-bounds} or  Bessel's inequality. Using
Theorem \ref{coefficient-bounds} we find 
\begin{equation}\label{tail-useful}
  \frac{1}{h(D)}\sum_{L\leq t_k\leq
  2L}\inprod{f\psi_t^2}{u_k}W_c(D,t_k)\ll\frac{(1+\abs{t})^{b+\varepsilon}}{(1+L)^{b}}\norm{\psi_t}_{\infty,K} D^{-1/12+\varepsilon}(1+L)^{2+\varepsilon}.
\end{equation}
It follows from Lemma \ref{eisenstein-bound-again} for the case of Eisenstein series and elementary considerations that
\begin{equation*}\norm{f\psi_t^2}_2\ll_f\norm{\psi_t}_{\infty,K}(1+\abs{t})^\varepsilon.\end{equation*}
Using  this and Bessel's inequality  we get the estimate
\begin{equation}\label{bulk-useful}
  \frac{1}{h(D)}\sum_{L\leq t_k\leq
  2L}\inprod{f\psi_t^2}{u_k}W_c(D,t_k)\ll \norm{\psi_t}_{\infty,K}(1+\abs{t})^{\varepsilon}  D^{-1/12+\varepsilon}(1+L)^{1+\varepsilon}.
\end{equation}
We have similar bounds for the continuous contribution. Using
\eqref{bulk-useful} for the bulk and \eqref{tail-useful} for the tail 
we find easily that for $V$ bounded away from zero we have 
\begin{align*}
 f\psi_t^2&-\frac{1}{\vol{\GmodH}}\int_{\GmodH}f\psi_t^2d\mu(z)\\
&\ll_{f,b}\norm{\psi_t}_{\infty,K} D^{-1/12+\varepsilon}((1+\abs{t})^{\varepsilon}V^{1+\varepsilon}+\frac{(1+\abs{t})^b}{V^b}V^{2+\e}),
\end{align*} when $b>2+\varepsilon$. Choosing
$V=(1+\abs{t})^{b/(b-1)}$ and $b$ sufficiently large we arrive at the result.
\end{proof}

\section{Proof of the main theorems}
We now have the necessary tools to prove  theorems \ref{maintheorem} and  \ref{maintheorem-conditional}.  We start by constructing appropriate test functions for the Selberg
pre-trace formula.
 \subsection{The pre-trace formula} This
 follows our previous investigations. We
 refer to \cite{PetridisRisager:2016a} for
additional details.
Let $\delta>0$ be a small parameter, which will eventually  be chosen to
depend on $X$ and $D$. Let 
\begin{equation*}
  k_\delta(u)=\frac{1}{4\pi\sinh^2(\delta/2)}1_{[0,\cosh\delta-1)/2]}(u),
\end{equation*} where $1_{A}(u)$ denotes the  indicator function of any set
$A$. Let $Y>0$  be defined by $\cosh(Y)={X}/{2}.$ This implies
that $4u(z,w)+2\leq X$ if and only if $d(z,w)\leq Y$, where $d(z,w)$
denotes the hyperbolic distance between $z$ and $w$. Consider now
\begin{equation*}
  k^{\pm}(u)=\left(1_{[0,(\cosh(Y\pm \delta)-1)/2]}*k_{\delta}\right)(u),
\end{equation*} where $*$ denotes the hyperbolic
convolution \begin{equation*}k_1*k_2(u(z,w))=\int_{\H} k_1(u(z,v))k_2(u(v,w))d\mu(v).\end{equation*}
With this choice of kernels it follows from the triangle
inequality for the hyperbolic distance
that 
\begin{equation*}
  k^-(u)\leq 1_{[0,(X-2)/4]}(u)\leq k^+(u),
\end{equation*}
see \cite[Eq. (5.4)]{PetridisRisager:2016a}.
By summing over $\g\in\G$ we find that
\begin{equation}\label{useful-inequality}K^+(z,w,X)\leq N(z,w,X)\leq K^+(z,w,X),\end{equation}
where
\begin{equation*}
K^\pm(z,w,X):=\sum_{\g\in\G}k^{\pm}(\g z,w,X).
\end{equation*}
Subtracting $\pi X/\vol{\GmodH}$ and averaging over Heegner points of
discriminant $D$, we see that  bounds  for  the absolute values of the two expressions 
\begin{equation*}
  \frac{1}{h(D)}\sum_{z\in\Lambda_D}f(z)\left(K^{\pm}(z,z,X)-\frac{\pi X}{\vol{\GmodH}}\right)
\end{equation*}
imply the same bound for  the absolute value of 
\begin{equation*} \frac{1}{h(D)}\sum_{z\in\Lambda_D}f(z)\left(N(z,z,X)-\frac{\pi X}{\vol{\GmodH}}\right).\end{equation*}
The advantage of approximating $N(z,w,X)$ by the automorphic kernels $K^\pm(z,w,X)$ is that,
contrary to  $1_{[0,(X-2)/4)]}$, the kernels $k^{\pm}$  are admissible
in the Selberg pre-trace formula. Moreover, the corresponding
Selberg--Harish-Chandra transforms $h^{\pm}$ can be computed
explicitly in terms of special functions, as  the Selberg--Harish-Chandra
transform maps the hyperbolic convolution $k_1*k_2$ into the product of the transforms $h_1h_2$, see \cite[p.~323]{Chamizo:1996b}. 

Let $h_R$ denote the Selberg--Harish-Chandra transform of
$k_R=1_{[0,(\cosh R -1)/2]},$ with $R>0$. The function $h_R$ can be computed: \begin{equation*}
h_R(t)=\sqrt{2\pi \sinh R}\Re\left(e^{itR}\frac{\Gamma(it)}{\Gamma(3/2+it)}F\left(\frac{-1}{2},\frac{3}{2},1-it,\frac{1}{(1-e^{2R})}\right)\right),
\end{equation*}
where $F$ is the Gauss hypergeometric function, see \cite[p.~321]{Chamizo:1996b}.
It follows from the series representation of
$F$ that, for large enough $R$, say  $R>\log(2)/2$, we have
\begin{equation*}F\left(\frac{-1}{2},\frac{3}{2},1-it,\frac{1}{(1-e^{2R})}\right)=1+O(e^{-2R}\min(1,\abs{t}^{-1}).\end{equation*}
For $R$ small, say less than $1$, and $t$ real it is known that, see e.g. \cite[Lemma 2.4 (c), p.~320]{Chamizo:1996b},
 \begin{equation*}
 h_R(t)=2\pi R^2\frac{J_1(Rt)}{Rt}\sqrt{\frac{\sinh
     R}{R}}+O(R^2\min(R^2,\abs{t}^{-2}). 
\end{equation*}
Here $J_1$ is the Bessel function of order 1. It satisfies
\begin{equation*}
  2\frac{J_1(x)}{x}=1_{[0,1]}(x)+O(\min(\abs{x},\abs{x}^{-3/2})),
\end{equation*}
see \cite[B.28, B.35]{Iwaniec:2002a}.
It is also convenient to use the uniform bound $h_R(t)=O((R+1)e^{R/2})$, see \cite[Lemma 2.4 (d), p.~320]{Chamizo:1996b}.

The group $\G=\pslz$ has no small eigenvalues so we only need to estimate
$h^{\pm}(t)$ at $t=i/2$ and for $t\in\R$.  We have
\begin{equation*}
  h^{\pm}(t)=h_{Y\pm \delta}(t)\frac{h_\delta(t)}{4\pi\sinh^2(\delta/2)}.
\end{equation*}
A direct computation \cite[Lemma 2.4 (d), p.~320]{Chamizo:1996b} shows
that
\begin{equation}\label{zero-eigenvalue}
  h^{\pm}(i/2)=2\pi(\cosh(Y\pm\delta)-1)\frac{2\pi(\cosh\delta
    -1)}{4\pi\sinh^2(\delta/2)}=\pi X+O(1+\delta X).
\end{equation}
To estimate $h^\pm(t)$ for $t$ real 
we combine the bounds above and find
\begin{align}
  \nonumber h^{\pm}(t)&=O\left(\frac{\sqrt{X}}{t^{3/2}}\left(\min(1,(\delta\abs{t})^{-3/2})+\min(\delta^2,\abs{t}^{-2})\right)
  \right)\\ 
\label{bound1} &=O\left(\frac{\sqrt{X}}{t^{3/2}}\left(\min(1,(\delta\abs{t})^{-3/2})\right)\right),
\end{align}
see \cite[Eq. (5.5), (5.10)]{PetridisRisager:2016a}. Finally  
\begin{equation} \label{bound2-trivial}
  h^\pm(t)=O(\sqrt{X}\log X),
\end{equation} where the last bound is uniform for $t$ real.

\subsection{Applying the pre-trace formula}
 By the
pre-trace formula we have 
\begin{align*}K^{\pm}(z,z,X)=\sum_{t_j}h^\pm(t_j)\abs{u_j(z)}^2
+\frac{1}{4\pi}\int_\R  h^{\pm}(t)
\abs{E(z,1/2+it)}^2dt.\end{align*}
Using \eqref{zero-eigenvalue} we therefore see that \begin{align}\label{too-long-to-type}\nonumber\frac{1}{h(D)}&\sum_{z\in\Lambda_D}f(z)\left(K^{\pm}(z,z,X)-\frac{\pi
      X}{\vol{\GmodH}}\right)\\
\nonumber 
 =&\sum_{t_j\in\R}h^\pm(t_j)\frac{1}{h(D)}\sum_{z\in\Lambda_D}
   f(z)\abs{u_j(z)}^2\\
\nonumber&+\frac{1}{4\pi}\int_\R
  h^{\pm}(t)
\frac{1}{h(D)}\sum_{z\in\Lambda_D}f(z)\abs{E(z,1/2+it)}^2dt+O(\delta                                                                                     X)\\
 =&\sum_{t_j}h^\pm(t_j)\frac{1}{\vol{\GmodH}}\int_{\GmodH} f\abs{u_j}^2d\mu       +Q_c(X, \delta, D)\\ 
\nonumber&+\frac{1}{4\pi}\int_\R
  h^{\pm}(t)
\frac{1}{\vol{\GmodH}}\int_{\GmodH}f\abs{E(\cdot,1/2+it)}^2d\mu\, dt +Q_E(X, \delta, D)+O(\delta X),\\ \nonumber
\end{align}
where 
\begin{align}\nonumber
 Q_c(X, \delta, D)=&\sum_{t_j\in \R}h^\pm(t_j)\left(\frac{1}{h(D)}\sum_{z\in\Lambda_D} f(z)\abs{u_j(z)}^2-\frac{1}{\vol{\GmodH}}\int_{\GmodH} f\abs{u_j}^2d\mu\right),
\\\label{flocke}
 Q_E(X, \delta, D)=&\frac{1}{4\pi}\int_\R
  h^{\pm}(t)\left(\frac{1}{{h(D)}}\sum_{z\in\Lambda_D}f(z)\abs{E(z,1/2+it)}^2
                     \right.\\  &\quad \quad\quad\quad\quad\quad\quad\left.-\frac{1}{\vol{\GmodH}}\int_{\GmodH}f\abs{E(\cdot,1/2+it)}^2d\mu\right)
dt. \nonumber
\end{align}
The first and third terms in \eqref{too-long-to-type} are exactly the expressions that are treated in
\cite[Sec 6.]{PetridisRisager:2016a}, where we found -- using several deep results from
\cite{LuoSarnak:1995}, e.g. \eqref{oscillation} -- that the first term is bounded by 
$O(X^{7/12+\varepsilon})$ as long as we assume that $\delta$ tends to
zero at least as fast as $X^{-c}$ for some $c>0$.
  Compare \cite[Lemmata 6.2 and 6.3]{PetridisRisager:2016a}. The
  third term is $O(X^{1/2+\varepsilon})$ by \cite[Lemma 6.1]{PetridisRisager:2016a}.

The second term $Q_c(X, \delta, D)$  in \eqref{too-long-to-type} is where we
need bounds on the sup-norm. We first notice that we have the trivial bound 
\begin{align*}
\sum_{T\leq t_j\leq 2T}\frac{1}{h(D)}\sum_{z\in\Lambda_D}
  f(z)\abs{u_j(z)}^2&=\frac{1}{h(D)}\sum_{z\in\Lambda_D} f(z)\sum_{T\leq
  t_j\leq 2T}\abs{u_j(z)}^2\\
& \ll_f \Big\Vert{\sum_{T\leq
  t_j\leq 2T}\abs{u_j(z)}^2}\Big\Vert_\infty .
\end{align*}
Since $\int_{\GmodH} f(z)\abs{u_j(z)}^2d\mu(z))\ll_f 1$ we easily find
from Proposition \ref{sup-of-averages} and Weyl's law \eqref{weyllaw}   that 
\begin{equation*}
 \sum_{T\leq t_j\leq 2T}\left(\frac{1}{h(D)}\sum_{z\in\Lambda_D}
  f(z)\abs{u_j(z)}^2-\frac{1}{\vol{\GmodH}}\int_{\GmodH}
  f\abs{u_j}^2d\mu\right)
  =O_f(T^2). 
\end{equation*}
Interpolating this -- using $\min{ (B,C)}\leq B^aC^{1-a}$ for $0\leq a\leq
1$ -- with the bound we get from bounding each term using
Theorem \ref{propo} and applying Theorem \ref{sup-norm-average-intro}  we find that for any $0\leq a \leq 1$
\begin{align*}
  \sum_{T\leq t_j\leq 2T}\left(\frac{1}{h(D)}\sum_{z\in\Lambda_D}
  f(z)\abs{u_j(z)}^2\right.&-\left.\frac{1}{\vol{\GmodH}}\int_{\GmodH}
  f\abs{u_j}^2d\mu\right)\\&=O_f(T^{2+a(3/2-1/8)+\varepsilon}D^{-a/12+\varepsilon})). 
\end{align*}

Using \eqref{bound2-trivial} for a bounded set of $t_j$'s, \eqref{bound1} in the ranges $1\leq|t_j|<\delta^{-1}$ and $|t_j|\geq\delta^{-1} $, dyadic decomposition,  and the estimate above, we find
that the  quantity $Q_c(X, \delta, D)$ in \eqref{flocke} is
$O(X^{1/2}\delta^{-(1/2+a(3/2-1/8)+\varepsilon)}D^{-a/12+\varepsilon})$, for any $0\leq a\leq 1$ satisfying  $a(3/2-1/8)+\varepsilon<1$.

The strategy for bounding the fourth term of \eqref{too-long-to-type} is
basically the same. We use Proposition
\ref{sup-of-averages-Eisenstein}, Lemma
\ref{eisenstein-bound-again}, Theorem \ref{propo} and \eqref{Youngs-bound} to see that 
\begin{align*}
 \int_T^{2T}
 \left(\frac{1}{h(D)}{\sum_{z\in\Lambda_D}f(z)\abs{E(z,1/2+it)}^2}-\frac{1}{\vol{\GmodH}}{\int_{\GmodH}f\abs{E(\cdot,1/2+it)}^2d\mu}\right)
d
t
\end{align*}
is $O_f(\min(T^2,
T^{5/2-1/8+\varepsilon}D^{-1/12+\varepsilon})=O(T^{2+(1/2-1/8)b+\varepsilon}D^{-b/12+\varepsilon})$
for any $0\leq b\leq 1$. Doing dyadic decomposition
we find that the fourth quantity $Q_E(X, \delta, D)$ is $
O(X^{1/2}\delta^{-(1/2+b(1/2-1/8)+\varepsilon)}D^{-b/12+\varepsilon})$. Chosing
$a=b$ we see that this term is
smaller than the third. 

Putting everything together we find that 
\begin{align*}
  \frac{1}{h(D)}\sum_{z\in\Lambda_D}f(z)&\left(K^{\pm}(z,z,X)-\frac{\pi
      X}{\vol{\GmodH}}\right)=\\
&O(\delta X+X^{7/12+\varepsilon} +X^{1/2}\delta^{-(1/2+a(3/2-1/8)+\varepsilon)}D^{-a/12+\varepsilon}).
\end{align*}
Choosing $a$  minimal ($a=0$) and balancing error terms we recover
Selberg's  bound $O(X^{2/3+\varepsilon})$ with no saving due to the averaging in $D$. If we
choose $a$ as large as allowed, i.e. close to $1/(3/2-1/8)=8/11$, the error 
is 
\begin{equation*}
  O(\delta X+X^{7/12+\varepsilon} +X^{1/2}\delta^{-(3/2+\varepsilon)}D^{-a/12+\varepsilon}).
\end{equation*}
 To balance the first and third term we choose
 $\delta=X^{-1/5}D^{-4/165}$, which gives Theorem  \ref{maintheorem}.

 \begin{proof}[Proof of Theorem \ref{maintheorem-conditional}]
We assume \ref{Lindelof-conjecture}, \ref{sup-norm-conjecture}, and
  \ref{spectral-average-conjecture}. Using Remark
  \ref{bound-on-GRH}  and \cite[Remark 6.4]{PetridisRisager:2016a}
 and the same technique as above, we find that for any $0\leq a<1$
\begin{align*}
  \frac{1}{h(D)}\sum_{z\in\Lambda_D}f(z)&\left(K^{\pm}(z,z,X)-\frac{\pi
      X}{\vol{\GmodH}}\right)=\\
&O(\delta X+X^{1/2+\varepsilon} +X^{1/2}\delta^{-(1/2+a(3/2-1/2)+\varepsilon)}D^{-a/4+\varepsilon}).
\end{align*}
We choose $a$ close to 1 and $\delta=X^{-1/5}D^{-1/10}$ to get  the result.
 \end{proof}

\bibliographystyle{plain}
\def\cprime{$'$}

\end{document}